\documentclass[english,oneside]{smfart}
\usepackage[no-math]{fontspec}
\usepackage{amsthm,mathtools,amssymb,textcomp,extarrows,bm,mleftright,graphicx,stmaryrd,scalerel,enumitem,float,booktabs,fixdif,fontawesome5}
\newcommand{\link}[1]{\textsuperscript{\href{#1}{\faIcon{external-link-alt}}}}
\usepackage[dvipsnames,svgnames]{xcolor}
\DeclareMathAlphabet{\mathbbb}{U}{bbold}{m}{n}
\mleftright

\makeatletter
\AtBeginDocument{%
 \let\glb@currsize\relax
}
\makeatother

\usepackage{mathrsfs}

\ExplSyntaxOn
\NewDocumentCommand{\definealphabet}{mmmm}
 {%
  \int_step_inline:nnn { `#3 } { `#4 }
   {
    \cs_new_protected:cpx { #1 \char_generate:nn { ##1 }{ 11 } }
     {
      \exp_not:N #2 { \char_generate:nn { ##1 } { 11 } }
     }
   }
 }
\ExplSyntaxOff
\definealphabet{bb}{\mathbb}{A}{Z}
\definealphabet{cal}{\mathcal}{A}{Z}
\definealphabet{frak}{\mathfrak}{A}{z}
\definealphabet{rm}{\mathrm}{A}{z}
\definealphabet{bf}{\mathbf}{A}{z}
\definealphabet{scr}{\mathscr}{A}{Z}
\definealphabet{tt}{\mathtt}{A}{Z}
\definealphabet{ul}{\underline}{A}{z}
\newcommand{\Gal}{\operatorname{Gal}}
\newcommand{\id}{\operatorname{id}}
\newcommand{\cinv}{^{\circ-1}}
\newcommand{\ev}{\operatorname{ev}}
\newcommand{\FT}{\operatorname{FT}}

\newcommand{\lto}{\longrightarrow}
\newfontface\logix {logix.otf}[Scale=1.0,NFSSFamily=logix]
\DeclareSymbolFont{Logix}{TU}{logix}{m}{n}
\newcommand{\lgl}[1]{\ensuremath{\Udelimiter 4 \symLogix "#1}}

\newcommand{\OpnDblParn}{\lgl{E31B}\relax}
\newcommand{\ClsDblParn}{\lgl{E35B}\relax}

\newcommand{\pparen}[1]{\left\OpnDblParn #1\right\ClsDblParn}
\newcommand{\bbrac}[1]{\left\llbracket #1\right\rrbracket}

\usepackage{hyperref}
\usepackage[nameinlink]{cleveref}
\Crefname{theorem}{Theorem}{Theorems}
\Crefname{introthm}{Theorem}{Theorems}
\Crefname{conjecture}{Conjecture}{Conjectures}
\Crefname{lemma}{Lemma}{Lemmas}
\Crefname{definition}{Definition}{Definitions}
\Crefname{remark}{Remark}{Remarks}
\Crefname{question}{Question}{Questions}
\Crefname{proposition}{Proposition}{Propositions}
\Crefname{corollary}{Corollary}{Corollaries}
\Crefname{equation}{}{}
\Crefname{item}{}{}
\Crefname{example}{Example}{Examples}
\Crefname{proof}{Proof}{Proofs}
\newlist{thmenum}{enumerate}{1}
\setlist[thmenum]{label=\arabic*., ref=\thetheorem~(\arabic*)}
\crefalias{thmenumi}{theorem}
\newlist{propenum}{enumerate}{1}
\setlist[propenum]{label=\arabic*., ref=\theproposition~(\arabic*)}
\crefalias{propenumi}{proposition}
\newlist{lemenum}{enumerate}{1}
\setlist[lemenum]{label=\arabic*., ref=\thelemma~(\arabic*)}
\crefalias{lemenumi}{lemma}
\AddToHook{env/introthm/before}{\addvspace{3pt}}
\AddToHook{env/introthm/after}{\addvspace{3pt}}
\AddToHook{env/theorem/before}{\addvspace{3pt}}
\AddToHook{env/theorem/after}{\addvspace{3pt}}
\AddToHook{env/example/before}{\addvspace{3pt}}
\AddToHook{env/example/after}{\addvspace{3pt}}
\AddToHook{env/definition/before}{\addvspace{3pt}}
\AddToHook{env/definition/after}{\addvspace{3pt}}
\AddToHook{env/proposition/before}{\addvspace{3pt}}
\AddToHook{env/proposition/after}{\addvspace{3pt}}
\AddToHook{env/lemma/before}{\addvspace{3pt}}
\AddToHook{env/lemma/after}{\addvspace{3pt}}
\AddToHook{env/remark/before}{\addvspace{3pt}}
\AddToHook{env/remark/after}{\addvspace{3pt}}
\AddToHook{env/corollary/before}{\addvspace{3pt}}
\AddToHook{env/corollary/after}{\addvspace{3pt}}

\usepackage[url=true,backend=biber,style=ext-alphabetic,hyperref=true,giveninits=true]{biblatex}
\makeatletter
\renewcommand*\subjclass[2][2020]{%
  \def\@subjclass{#2}%
  \@ifundefined{subjclassname@#1}{%
    \ClassWarning{\@classname}{Unknown edition (#1) of Mathematics
      Subject Classification; using '2020'.}%
  }{%
    \@xp\let\@xp\subjclassname\csname subjclassname@#1\endcsname
  }%
}
\@namedef{subjclassname@1991}{%
  \textup{1991} Mathematics Subject Classification}
\@namedef{subjclassname@2000}{%
  \textup{2000} Mathematics Subject Classification}
\@namedef{subjclassname@2010}{%
  \textup{2010} Mathematics Subject Classification}
\@namedef{subjclassname@2020}{%
  \textup{2020} Mathematics Subject Classification}
\@xp\let\@xp\subjclassname\csname subjclassname@2020\endcsname
\makeatother

\newtheorem{theorem}{Theorem}[section]

\AddToHook{env/example/begin}{\crefalias{theorem}{example}}
\newtheorem{lemma}[theorem]{Lemma}
\AddToHook{env/lemma/begin}{\crefalias{theorem}{lemma}}
\newtheorem{remark}[theorem]{Remark}
\AddToHook{env/remark/begin}{\crefalias{theorem}{remark}}
\newtheorem{proposition}[theorem]{Proposition}
\AddToHook{env/proposition/begin}{\crefalias{theorem}{proposition}}
\newtheorem{definition}[theorem]{Definition}
\AddToHook{env/definition/begin}{\crefalias{theorem}{definition}}

\AddToHook{env/conjecture/begin}{\crefalias{theorem}{conjecture}}

\AddToHook{env/corollary/begin}{\crefalias{theorem}{corollary}}

\AddToHook{env/question/begin}{\crefalias{theorem}{question}}

\AddToHook{env/assumption/begin}{\crefalias{theorem}{assumption}}

\newtheorem{introthm}{Theorem}

\Crefformat{enumi}{#2\textup{(#1)}#3}

\usepackage[verbose=silent]{microtype}
\usepackage{geometry}

\let\opn\operatorname

\numberwithin{equation}{section}

\title{Non-liftability of the field of norms of the $p$-adic false Tate curve extension}
\usepackage{orcidlink}
\author{Yijun Yuan\orcidlink{0000-0001-6571-6980}}
\address{Institute for Theoretical Sciences, Westlake University, No. 600 Dunyu Road, Sandun town, Xihu district, Hangzhou, Zhejiang Province, 310030, China}
\email{941201yuan@gmail.com}
\urladdr{https://yijunyuan.github.io/}

\begin{document}
\subjclass{11S15, 11S20, 11S31, 13F25, 68V20}
\keywords{field of norms, false Tate curve extension, liftability, ramification}
\begin{abstract}
	Let $p\geq3$ be a prime number. We prove that the action of the Galois group on the field of norms of the $p$-adic false Tate curve extension $\bfQ_p\left(\zeta_{p^\infty},p^{1/p^\infty}\right)$ of $\bfQ_p$ is not liftable to its Cohen ring with respect to any admissible coefficient field.
\end{abstract}

\maketitle
\tableofcontents

\section{Introduction}\label{sec:85117}

Let $p\geq 3$ be a prime number. Let $K$ be a finite extension of $\bfQ_p$ with residue field $\bfF_q$, and let $L$ be an infinite totally ramified Galois extension of $K$ whose Galois group is a $p$-adic Lie group. By a theorem of Sen, such an extension is strictly arithmetically profinite (cf. \cite[\S1]{bergerLiftingFieldNorms2014}), and Fontaine and Wintenberger \cite{fontaineExtensionsAlgebriquesCorps1979,fontaineCorpsNormesCertaines1979,wintenbergerCorpsNormesCertaines1983} attach to it the \emph{field of norms} $X_K(L)$. This is a complete discrete valuation field of characteristic $p$ with residue field $\bfF_q$, and $\Gamma\coloneqq\Gal(L/K)$ acts continuously on $X_K(L)$ by field automorphisms. Moreover, the absolute Galois group of $X_K(L)$ is canonically isomorphic to that of $L$.

For the cyclotomic extension $\bfQ_p^{\opn{cyc}}\coloneqq\bigcup_n \bfQ_p(\mu_{p^n})$, Fontaine \cite{Fontaine2007} lifted this structure to characteristic $0$: the $p$-Cohen ring $\bfA_{\bfQ_p}$ of $X_{\bfQ_p}(\bfQ_p^{\opn{cyc}})$ carries a lifted Frobenius map $\varphi$ and a continuous action of $\Gamma=\Gal(\bfQ_p^{\opn{cyc}}/\bfQ_p)$, which commute with each other. This allowed him to prove the equivalence between the category of $\bfZ_p$-representations of $\Gal(\overline{\bfQ}_p/\bfQ_p)$ and the category of \'etale $(\varphi,\Gamma)$-modules over $\bfA_{\bfQ_p}$. This equivalence, which can be generalized to any finite extension $K$ of $\bfQ_p$, has become a fundamental tool in $p$-adic Hodge theory, and is a crucial ingredient in Colmez's construction of the $p$-adic local Langlands correspondence for $\opn{GL}_2(\bfQ_p)$ (cf. \cite{AST_2010__330__281_0}).

\subsection{Liftability of the field of norms}
To establish the $p$-adic local Langlands correspondence beyond $\opn{GL}_2(\bfQ_p)$, an analogous theory of $(\varphi,\Gamma)$-modules for the $p$-adic Lie extensions $L$ of $K$ other than the cyclotomic one is desired. Such a theory exists as soon as the action of $\Gamma$ on $X_K(L)$ can be lifted together with a commuting Frobenius. Berger formulated this condition as follows.
\begin{definition}[cf. {\cite[Question 1.2]{bergerLiftingFieldNorms2014}}]\label{def:49607}
	Let $F$ be a finite extension of $\bfQ_p$ with uniformizer $\varpi_F$ and the same residue field $\bfF_q$ as $K$. Let\footnote{We emphasize that the definition of $\bfA_K$ simultaneously depends on $K$, $L$ and $F$. We omit $L$ and $F$ from the notation for simplicity when they are clear from the context. } $\bfA_K$ be the $\varpi_F$-adic completion of $\calO_F\bbrac{T}[1/T]$, which is a complete discrete valuation ring with uniformizer $\varpi_F$ and residue field $\bfF_q\pparen{T}$. We say that the action of $\Gamma\coloneqq \Gal(L/K)$ on $X_K(L)$ is \textbf{$F$-liftable} if there exist elements $P$ and $(F_g)_{g\in\Gamma}$ in $\bfA_K$ such that
	\begin{enumerate}
		\item $\overline{P}=T^q$ and $\opn{ev}_u(\overline{F_g})=g\ast u$ for every $g\in\Gamma$, where $\overline{P}$ and $\overline{F_g}$ denote the reductions of $P$ and $F_g$ modulo $\varpi_F$, $u$ is a uniformizer of $X_K(L)$, and $\opn{ev}_u\colon \bfF_q\pparen{T}\lto X_K(L)$ is the isomorphism that sends $T$ to $u$.
		\item $F_g\circ P=P\circ F_g$ and $F_h\circ F_g=F_{gh}$ for all $g,h\in\Gamma$.
	\end{enumerate}
\end{definition}
\begin{remark}
	We will show in \Cref{prop:71074} that the $F$-liftability of the action of $\Gamma$ on $X_K(L)$ is independent of the choice of uniformizer $u$.
\end{remark}
Berger proved in \cite[Theorem B]{bergerLiftingFieldNorms2014} that if the action of $\Gal(L/K)$ on $X_K(L)$ is $F$-liftable with $P\in\calO_F\bbrac{T}$, then $\Gal(L/K)$ is abelian. Poyeton \cite[Theorem 0.1]{poyetonFormalGroupsLifts2022} showed that, under the same assumption on $P$, if moreover $F/\bfQ_p$ is Galois and contains $K$, then there exist a finite extension $K'$ of $\bfQ_p$ and a relative Lubin--Tate formal group $S$ over a subfield of $K'$ such that $L$ is contained in the extension of $K'$ generated by the torsion points of $S$, with finite index.

\subsection{The false Tate curve extension}\label{subsec:93286}

The results just quoted assume that the lifted Frobenius $P$ is a power series, which excludes many interesting $p$-adic Lie extensions. The simplest non-commutative example is the \textbf{$p$-adic false Tate curve extension} of $\bfQ_p$,
$$\bfQ_p^{\opn{FT}}\coloneqq \bigcup_n \bfQ_p\left(\zeta_{p^n},p^{1/p^n}\right),$$
where $\left(\zeta_{p^n}\right)_{n\geq1}$ and $\left(p^{1/p^n}\right)_{n\geq1}$ are fixed compatible systems of primitive $p^n$-th roots of unity and of $p^n$-th roots of $p$. This is a totally ramified $2$-dimensional $p$-adic Lie extension of $\bfQ_p$ with non-commutative Galois group. As the normal closure of the Kummer extension $\bfQ_p^{\opn{Kum}}\coloneqq\bigcup_n \bfQ_p\left(p^{1/p^n}\right)$, it appears in various works on (integral) $p$-adic Hodge theory, and serves as a basic test-case in non-commutative Iwasawa theory. We summarize the structural results on the Galois group of the false Tate curve extension as follows.
\begin{lemma}[cf. {\cite[Corollary 6.2]{vivianiRamificationGroupsArtin2004}, \cite[Lemma 5.1.2]{liuLatticesSemistableRepresentations2008}}]\label{lem:23436}
	\begin{enumerate}
		\item The extension $\bfQ_p^{\FT}/\bfQ_p$ is totally ramified.
		\item One has isomorphisms of Galois groups
		      $$\Gal\left(\bfQ_p^{\opn{FT}}\middle/\bfQ_p\right)\cong \Gal\left(\bfQ_p^{\opn{FT}}\middle/\bfQ_p^{\opn{cyc}}\right)\rtimes\Gal\left(\bfQ_p^{\opn{FT}}\middle/\bfQ_p^{\opn{Kum}}\right),$$
			  $$\Gal\left(\bfQ_p^{\opn{FT}}\middle/\bfQ_p^{\opn{Kum}}\right)\cong\Gal\left(\bfQ_p^{\opn{cyc}}\middle/\bfQ_p\right)\cong\bfZ_p^\times \text{ and }\Gal\left(\bfQ_p^{\opn{FT}}\middle/\bfQ_p^{\opn{cyc}}\right)\cong\bfZ_p.$$
		      By abuse of notation, we identify the groups on both sides of the isomorphisms above.
		\item Let $\tau$ be the topological generator of $\Gal\left(\bfQ_p^{\opn{FT}}\middle/\bfQ_p^{\opn{cyc}}\right)$ that sends $p^{1/p^n}$ to $\zeta_{p^n}\cdot p^{1/p^n}$  for every $n\geq1$,
		      and let $\delta$ be a generator of the torsion subgroup of
		      $$\Gal\left(\bfQ_p^{\opn{FT}}\middle/\bfQ_p^{\opn{Kum}}\right)\cong\bfZ_p^\times\cong\bfZ_p\times\left(\bfZ/p\bfZ\right)^\times.$$
		      Then one has $\delta\cdot\tau=\tau^{\chi(\delta)}\cdot\delta$, where $\chi$ is the $p$-adic cyclotomic character.
	\end{enumerate}
\end{lemma}

Note that every uniformizer of the field of norms $X_{\bfQ_p}\left(\bfQ_p^{\FT}\right)$ is a norm-compatible family of uniformizers of the finite layers $\bfQ_p\left(p^{1/p^m},\zeta_{p^n}\right)$. An explicit construction of such a family would give a direct description of the action of $\Gal\left(\bfQ_p^{\FT}\middle/\bfQ_p\right)$ on the field of norms, and hence a concrete approach to the liftability problem (cf. \cite{BergerComment}, \cite[page 780]{vivianiRamificationGroupsArtin2004}, \cite[Remark 4.5]{caisCharacterizationStrictlyAPF2016}). Partial results in this direction are obtained in \cite{bellemareExplicitUniformizersCertain2020,wangUniformizerFalseTate2021,wangUniformizerFalseTate2024}.

\subsection{Main results}\label{subsec:27566}

The main result of this article is a negative answer to the liftability problem for the false Tate curve extension over every coefficient field allowed by \Cref{def:49607}. Since its field of norms has residue field $\bfF_p$, these coefficient fields are precisely the finite totally ramified extensions of $\bfQ_p$.%

\begin{introthm}\label{thm:84624}
	Let $p\geq3$. For every finite totally ramified extension $F/\bfQ_p$, the action of $\Gal\left(\bfQ_p^{\FT}\middle/\bfQ_p\right)$ on the field of norms $X_{\bfQ_p}\left(\bfQ_p^{\FT}\right)$ is not $F$-liftable.
\end{introthm}
In fact, the obstruction already occurs at a finite level:
\begin{proposition}\label{thm:42866}
	We use the same notation and assumptions as in \Cref{def:49607}, \Cref{lem:23436} and \Cref{thm:84624}. There do not exist elements $P,F_\tau,F_{\delta}$ in the ring
	$$\bfA_{\bfQ_p}/\varpi_F^{v_{\varpi_F}(p)+1}\bfA_{\bfQ_p}\cong \left(\calO_F/\varpi_F^{v_{\varpi_F}(p)+1}\right)\pparen{T}$$
	such that their reductions to $\bfF_p\pparen{T}$ satisfy the conditions that
	$$\overline P=T^p,\ \opn{ev}_u(\overline{F_\tau})=\tau\ast u,\ \opn{ev}_u(\overline{F_{\delta}})=\delta\ast u,$$
	$F_\tau\circ P=P\circ F_\tau$ and $F_\delta\circ P=P\circ F_\delta$.
	\end{proposition}

\subsection{Motivation and strategy of the proof}\label{subsec:29440}

Our approach to \Cref{thm:42866} is inspired by the obstruction theory of Nori and Srinivas \cite[Appendix]{CM_1987__64_2_191_0} (following Deligne and Illusie \cite{DI87}) for lifting smooth varieties and their morphisms compatibly with Frobenius. Specifically, \cite[Appendix, Proposition~3]{CM_1987__64_2_191_0} studies when a morphism $\varphi\colon X_1\to X_2$ between smooth varieties over a perfect field $k$ of characteristic $p$, equipped with fixed flat liftings to $W_2(k)$ and liftings of Frobenius, can be lifted to $W_2(k)$ compatibly with the Frobenius lifts. The obstruction lies in $\rmH^0\left(X_1,\varphi^*T_{X_2}\otimes_{\calO_{X_1}}B^1_{X_1}\right)$, where $B^1_{X_1}=\opn{im}\left(\d\colon\calO_{X_1}\to\Omega^1_{X_1/k}\right)$ is the sheaf of exact $1$-forms, endowed with the $\calO_{X_1}$-module structure $a\cdot\d b=a^p\d b$.

In our setting, we adapt the local ring calculation behind \cite[Proposition 3 of Appendix]{CM_1987__64_2_191_0} to continuous substitutions of Laurent series. Consider first the unramified case $F=\bfQ_p$. Fix a tentative Frobenius lift $P\in(\bfZ/p^2\bfZ)\pparen{T}$ with $\overline P=T^p$, and let $f\in T\bfF_p\bbrac{T}$ be the series attached to an element in the Galois group. For any lift $G\in(\bfZ/p^2\bfZ)\pparen{T}$ of $f$, both $G\circ P$ and $P\circ G$ reduce to $f^p$ modulo $p$, so their difference is divisible by $p$. Modifying the choice of lift $G\mapsto G+pU$ changes the series $\overline{(G\circ P-P\circ G)/p}\in\bfF_p\pparen{T}$ by the $p$-th power $\overline U^p$. Applying the formal derivative $\d/\d T$ kills $\bfF_p\pparen{T}^p$---passing to exact differentials as in Nori--Srinivas---and removes the dependence on $G$. By the chain rule, the vanishing of this differential obstruction translates directly into the characteristic-$p$ functional equation
\begin{equation}\label{eq:96406}
	h_P\circ f=(f')^{p-1}h_P,
\end{equation}
where $h_P\coloneqq\overline{P'/p}\in\bfF_p\pparen{T}^\times$. Whereas Nori and Srinivas work with fixed Frobenius lifts to lift a single morphism, in our problem the Frobenius lift $P$ is arbitrary and unknown. The key advantage of \eqref{eq:96406} is that it involves $P$ only through the single series $h_P$, which must simultaneously satisfy \eqref{eq:96406} for all lifted automorphisms in $\Gal\left(\bfQ_p^{\FT}\middle/\bfQ_p\right)$.

When the coefficient field $F/\bfQ_p$ is ramified with ramification index $e=v_{\varpi_F}(p)>1$, working modulo $\varpi_F^2$ is no longer sufficient because $(T^p)'=pT^{p-1}$ vanishes modulo $\varpi_F^e$. Instead, in \Cref{sec:29168}, we work in $(\calO_F/\varpi_F^{e+1})\pparen{T}$ and observe that $P'$ cannot vanish identically: it has a well-defined first nonzero $\varpi_F$-adic layer at some level $s(P)\in\{1,\ldots,e\}$. Extracting this layer defines the \emph{normalized derivative} $h_P\coloneqq\overline{P'/\varpi_F^{s(P)}}\in\bfF_p\pparen{T}^\times$, which still satisfies \eqref{eq:96406} for every commuting lift (cf. \Cref{lem:67522}). Depending on whether the term $pT^{p-1}$ contributes to this leading layer, the analysis splits into two cases:
\begin{itemize}
	\item \textbf{Case of $s(P)=e$ (cf. \Cref{sec:61843}):} Here $h_P=cT^{p-1}+A'$ for some $c\in\bfF_p^\times$ and $A\in\bfF_p\pparen{T}$, which forces $h_P'\neq 0$ and $v_T(h_P')\leq p-2$. For the Kummer iterates $\sigma_k=\tau^{p^k}$, we use the rapid growth of their ramification numbers $d_k$ to conjugate $\sigma_k$ into a normal form (cf. \Cref{thm:41759}) and deduce from \eqref{eq:96406} that $v_T(h_P)\leq p-1-d_k$ (cf. \Cref{thm:10556}). Letting $k\to\infty$ yields a contradiction.
	\item \textbf{Case of $s(P)<e$ (cf. \Cref{sec:42840}):} Here the contribution of $pT^{p-1}$ drops out, so $h_P=A'$ and $h_P\d T$ is an exact differential. In this case, we use the tame generator $\delta$ to normalize $h_P$ and construct a nonzero $\tau$-invariant exact differential $\omega$ (cf. \Cref{lem:45874}). An approximate-invariant argument (cf. \Cref{prop:15106}) then shows that the $p$-adic limit of $d_k$ as $k\lto\infty$ must be the integer $-\opn{ord}_T(\omega)-1$, contradicting the actual $p$-adic limit $\lim_{k\to\infty} d_k=(p-1)/(p+1)\notin\bfZ$.
\end{itemize}
In \Cref{sec:93558}, we compute the ramification numbers $d_k$ of $\tau^{p^k}$ and the linear coefficient of the tame generator $\delta$ for the false Tate curve extension, and in \Cref{sec:91137} we combine the two obstructions with these arithmetic inputs to prove \Cref{thm:42866} and deduce \Cref{thm:84624}.

\begin{remark}
	On can also consider the liftability problem for the false Tate curve extension
	$$K^{\opn{FT}}\coloneqq\bigcup_n K\left(\varpi_K^{1/p^n},\zeta_{p^n}\right)$$
	of a finite extension $K$ of $\bfQ_p$, where $\varpi_K$ is a uniformizer of $K$. Our strategy towards \Cref{thm:84624}, although can be slightly generalized, does not work in the full generality. For general $K$, we no longer have the explicit ramification calcuation in \cite{vivianiRamificationGroupsArtin2004}, which is crucial in our proof. On the other hand, the torsion subgroup of $\Gal\left(K^{\opn{FT}}/K\left(\varpi_K^{1/p^\infty}\right)\right)$ may vanish for large enough $K$.
	
	Despite these obstructions, we expect that the analogue of \Cref{thm:84624} holds for any finite extension $K$ of $\bfQ_p$ with respect to any admissible coefficient field.
\end{remark}

\subsection{Notation}\label{subsec:13146}
To make the rest of the article more readable, we summarize the notation that will be used throughout the article.
\begin{itemize}
	\item The field $F$ will be a finite totally ramified extension of $\bfQ_p$ with uniformizer $\varpi_F$. Let $e=v_{\varpi_F}(p)$.
	\item Let $\tau$ and $\delta$ be the elements of $\Gamma\coloneqq\Gal\left(\bfQ_p^{\FT}\middle/\bfQ_p\right)$ defined in \Cref{lem:23436}.
	\item Write $\bfE_{\bfQ_p}$ for the field of norms $X_{\bfQ_p}\left(\bfQ_p^{\FT}\right)$, and let $u$ be a fixed uniformizer of $\bfE_{\bfQ_p}$. The valuation $v_{\bfE}$ on $\bfE_{\bfQ_p}$ is normalized so that any uniformizer of it has valuation $1$.
	\item We equip $\bfF_p\pparen{T}$ with the $T$-adic valuation $v_T$ such that $v_T(T)=1$.
	\item For any uniformizer $\pi$ of $\bfE_{\bfQ_p}$ and for any $g\in\Gamma$, let $\opn{ev}_\pi\colon \bfF_p\pparen{T}\lto \bfE_{\bfQ_p}$ be the isomorphism that sends $T$ to $\pi$, and let $f_{g,\pi}\in T\bfF_p\bbrac{T}$ be the series with $\opn{ev}_{\pi}(f_{g,\pi})=g\ast \pi$. Since $g\ast \pi$ is again a uniformizer, the coefficient of $T$ in $f_{g,\pi}$ is nonzero.
	\item Let $\bfA_{\bfQ_p}$ be the $\varpi_F$-adic completion of $\calO_F\bbrac{T}[1/T]$, which is a complete discrete valuation ring with uniformizer $\varpi_F$ and residue field $\bfF_p\pparen{T}$. Set
	      $$\bfR_{\bfQ_p}\coloneqq\bfA_{\bfQ_p}/\varpi_F^{e+1}\bfA_{\bfQ_p}\cong \left(\calO_F/\varpi_F^{e+1}\right)\pparen{T}.$$
	      By the reduction $\overline{F}$ of an element $F$ of $\bfR_{\bfQ_p}$, we mean its image in $\bfF_p\pparen{T}$ under the natural projection.
\end{itemize}

\subsection{Lean 4 formalization}
All the results in this article, as well as their dependencies, have been fully formalized in Lean 4 (cf. \cite{mouraLean4TheoremProver2021}, \href{https://github.com/leanprover/lean4/releases/tag/v4.35.0}{version 4.35.0}\link{https://github.com/leanprover/lean4/releases/tag/v4.35.0}) using the mathlib4 library (cf. \cite{mathlib2020}, \href{https://github.com/leanprover-community/mathlib4/commit/1e75e9dd53d483fa01eff9ac7d4ea0472ce9d1fa}{commit \texttt{1e75e9d}}\link{https://github.com/leanprover-community/mathlib4/commit/1e75e9dd53d483fa01eff9ac7d4ea0472ce9d1fa}). The formalization is organized into 3 parts for reusability:
\begin{enumerate}
	\item \href{https://github.com/YijunYuan/FontaineWintenberger/tree/7b9f6b3459c69b3282cd09ced55d15d89344c58e}{\texttt{FontaineWintenberger}}\link{https://github.com/YijunYuan/FontaineWintenberger/tree/7b9f6b3459c69b3282cd09ced55d15d89344c58e}: this repository formalizes the results from \cite{wintenbergerCorpsNormesCertaines1983} that we need here: strictly arithmetically profinite extensions, the field of norms, and its properties.
	\item \href{https://github.com/YijunYuan/FalseTateExtension/tree/ae06abba8ed93ec6f8cef4369fed870970faec0c}{\texttt{FalseTateExtension}}\link{https://github.com/YijunYuan/FalseTateExtension/tree/ae06abba8ed93ec6f8cef4369fed870970faec0c}: this repository formalizes the definition and basic properties of the false Tate curve extension, including \Cref{lem:23436} and the ramification calculation in \cite[Corollary 6.10]{vivianiRamificationGroupsArtin2004} that we need in \Cref{prop:78229}.
	\item \href{https://github.com/YijunYuan/FalseTateNonliftability/tree/1762dd5e233404daefd430eddacd32b35d42a719}{\texttt{FalseTateNonliftability}}\link{https://github.com/YijunYuan/FalseTateNonliftability/tree/1762dd5e233404daefd430eddacd32b35d42a719}: this repository formalizes the body of this article.%
\end{enumerate}
The formalization is \texttt{sorry}-free and depends only on the standard axioms \texttt{Classical.choice}, \texttt{propext} and \texttt{Quot.sound} of Lean.

The initial version of the formalization was carried out by \href{https://github.com/frenzymath/Archon-Horizon}{\texttt{Archon-Horizon}}\link{https://github.com/frenzymath/Archon-Horizon}, an agentic auto-formalization system developed by the AI4Math team at BICMR, Peking University, with \href{https://deploymentsafety.openai.com/gpt-6-1-sol/gpt-6-1-sol.pdf}{\texttt{GPT-6.1 Sol}}\link{https://deploymentsafety.openai.com/gpt-6-1-sol/gpt-6-1-sol.pdf} used as the base model. Later, the code was thoroughly refactored by \href{https://anthropic.com/claude-fable-5-1-mythos-5-1-system-card}{\texttt{Claude Fable 5.1}}\link{https://anthropic.com/claude-fable-5-1-mythos-5-1-system-card} following a human-written plan. Finally, the code was polished by \href{https://deploymentsafety.openai.com/gpt-6-astra/gpt-6-astra.pdf}{\texttt{GPT-6 Astra}}\link{https://deploymentsafety.openai.com/gpt-6-astra/gpt-6-astra.pdf}, following the official Mathlib style guides\link{https://leanprover-community.github.io/contribute/style.html}\link{https://leanprover-community.github.io/contribute/naming.html}\link{https://leanprover-community.github.io/contribute/doc.html} as the reference.

\subsection*{Declaration of AI usage}
We use aforementioned LLMs to polish the English writing of this paper, to check the mathematical statements and proofs, and to assist in the Lean 4 formalization. We take full responsibility for the content of this paper and the formalization.

\subsection*{Acknowledgements}
The research is supported by the National Key R\&D Program of China (Grant No. 2024YFA1014000).

\section{Finite-precision lifts and the normalized derivative}\label{sec:29168}

\subsection{Substitution and independence of the uniformizer}\label{subsec:58949}

We first record the basic properties of series substitution in $\bfR_{\bfQ_p}$ and verify that liftability is independent of the choice of the uniformizer of $\bfE_{\bfQ_p}$. Since the verifications of the following properties of substitution are standard, we omit them.

For substitution, we use the Laurent-series topology on each $(\calO_F/\varpi_F^N\calO_F)\pparen{T}$ ($N\geq1$), with a neighborhoods basis $\{T^m(\calO_F/\varpi_F^N\calO_F)\bbrac{T}\}_{m\in\bfZ}$ at zero, and the induced inverse-limit topology on $\bfA_{\bfQ_p}$.

\begin{lemma}\label{lem:32175}
	Let $G\in\bfR_{\bfQ_p}$ have nonzero reduction $\overline G\in\bfF_p\pparen{T}$ with $v_T(\overline G)\geq 1$.
	\begin{lemenum}
		\item\label{lem:22046} For any $Q=\sum_{i\in\bfZ} a_iT^i\in\bfR_{\bfQ_p}$, the series $Q\circ G\coloneqq\sum_{i\in\bfZ} a_iG^i$ converges in $\bfR_{\bfQ_p}$. The map $Q\longmapsto Q\circ G$ is the unique continuous $(\calO_F/\varpi_F^{e+1}\calO_F)$-algebra endomorphism of $\bfR_{\bfQ_p}$ sending $T$ to $G$; we call it \textbf{substitution of the variable by $G$}.
		\item\label{lem:42740} If $H\in\bfR_{\bfQ_p}$ also satisfies $v_T(\overline H)\geq 1$, then $(Q\circ G)\circ H=Q\circ(G\circ H)$ and $(Q\circ G)'=(Q'\circ G)G'$.
		\item\label{lem:71399} If $H\in T(\calO_F/\varpi_F^{e+1}\calO_F)\bbrac{T}$ has invertible linear coefficient, then its formal compositional inverse $H\cinv$ belongs to $T(\calO_F/\varpi_F^{e+1}\calO_F)\bbrac{T}$, and $\overline{H\cinv}=\overline H\cinv$.
	\end{lemenum}
\end{lemma}

For any $G$ as in \Cref{lem:32175} and any $n\in\bfN$, we denote by $G^{\circ n}$ its $n$-fold compositional iterate, with the convention $G^{\circ0}=T$. The same substitution formalism applies over $\bfA_{\bfQ_p}$ and its quotients modulo any power of $\varpi_F$.

For later calculations in $\bfE_{\bfQ_p}$, we next summarize the basic properties of the coordinate series $f_{g,\pi}$ in characteristic $p$.

\begin{lemma}\label{lem:49963}
	Let $g,h\in\Gamma$ and let $\pi$ be a uniformizer of $\bfE_{\bfQ_p}$.
	\begin{lemenum}
		\item\label{lem:16018} The action of $g$ preserves $v_{\bfE}$. In particular, $f_{g,\pi}\in T\bfF_p\bbrac{T}$ has nonzero linear coefficient.
		\item\label{lem:21745} For every $a\in\bfF_p\pparen{T}$, one has
		\begin{equation}\label{eq:27955}
			g\ast\ev_\pi(a)=\ev_\pi(a\circ f_{g,\pi}).
		\end{equation}
		\item\label{lem:31185} One has $f_{gh,\pi}=f_{h,\pi}\circ f_{g,\pi}$, and hence $f_{g^n,\pi}=f_{g,\pi}^{\circ n}$ for all $n\geq0$.
		\item\label{lem:23769} If $\pi'=\ev_\pi(V)$ is another uniformizer with $V\in T\bfF_p\bbrac{T}$, then
		$$f_{g,\pi'}=V\circ f_{g,\pi}\circ V\cinv.$$
	\end{lemenum}
\end{lemma}

\begin{proof}
	The valuation $v_{\bfE}$ on the field of norms is induced by the normalized valuations on the finite layers of $\bfQ_p^{\FT}/\bfQ_p$. Since each finite-layer valuation is preserved by the Galois action, $g$ preserves $v_{\bfE}$, which gives \Cref{lem:16018}. Because $g$ acts continuously on $\bfE_{\bfQ_p}$ and fixes $\bfF_p$, moving $g$ inside the series $a$ gives \eqref{eq:27955}. Applying \eqref{eq:27955} to $a=f_{h,\pi}$ gives
	$$\ev_\pi(f_{gh,\pi})=(gh)\ast\pi=g\ast(h\ast\pi)=g\ast\ev_\pi(f_{h,\pi})=\ev_\pi(f_{h,\pi}\circ f_{g,\pi}),$$
	and applying $\ev_\pi^{-1}$ yields \Cref{lem:31185}. Finally, if $\pi'=\ev_\pi(V)$ is another uniformizer, then by definition of substitution we have $\ev_{\pi'}(Q)=\ev_\pi(Q\circ V)$ for all $Q\in\bfF_p\pparen{T}$. Applying \eqref{eq:27955} to $a=V$ gives
	$$\ev_\pi(V\circ f_{g,\pi})=g\ast\ev_\pi(V)=g\ast\pi'=\ev_{\pi'}(f_{g,\pi'})=\ev_\pi(f_{g,\pi'}\circ V).$$
	Applying $\ev_\pi^{-1}$ gives $V\circ f_{g,\pi}=f_{g,\pi'}\circ V$, which implies \Cref{lem:23769}.
\end{proof}

We now show that simultaneous conjugation transports commuting lifts across different choices of uniformizer of the field of norms.

\begin{proposition}\label{prop:71074}
	In the setting of \Cref{def:49607}, $F$-liftability is independent of the uniformizer of $X_K(L)$ used to identify it with $\bfF_q\pparen{T}$. The same holds modulo every power of $\varpi_F$, including for the existence of a Frobenius lift commuting with lifts of any prescribed collection of Galois elements.
\end{proposition}

\begin{proof}
	We prove this for the false Tate curve extension, as the same argument works verbatim for general $p$-adic Lie extension.

	Let $\pi$ and $\pi'$ be two uniformizers of $\bfE_{\bfQ_p}$, and write $\pi'=\ev_\pi(V)$ with $V\in T\bfF_p\bbrac{T}$ having nonzero linear coefficient. Choose a coefficientwise lift $H\in T\calO_F\bbrac{T}$ of $V$. Since the linear coefficient of $H$ lies in $\calO_F^\times$, its formal compositional inverse $H\cinv$ also lies in $T\calO_F\bbrac{T}$ with $\overline{H\cinv}=V\cinv$ by \Cref{lem:71399}. Given lifts $P$ and $(F_g)_{g\in\Gamma}$ in $\bfA_{\bfQ_p}$ with respect to $\pi$, define
	$$P_{\pi'}\coloneqq H\circ P\circ H\cinv,\ F_{g,\pi'}\coloneqq H\circ F_g\circ H\cinv.$$
	Since $V$ has coefficients in $\bfF_p$, we have $T^p\circ V = V^p = V\circ T^p$, so $\overline{P_{\pi'}}=V\circ T^p\circ V\cinv=T^p$. Moreover, by \Cref{lem:23769}, one has $\overline{F_{g,\pi'}}=V\circ f_{g,\pi}\circ V\cinv=f_{g,\pi'}$. By associativity of substitution, conjugation by $H$ preserves the relations $F_g\circ P=P\circ F_g$ and $F_h\circ F_g=F_{gh}$. By symmetry (conjugating by $H\cinv$), $F$-liftability with respect to $\pi$ and $\pi'$ are equivalent. The exact same conjugation argument applies verbatim modulo any power of $\varpi_F$ and for any prescribed subset of $\Gamma$.
\end{proof}

\subsection{The normalized derivative and the level dichotomy}\label{subsec:10624}

We now formalize the extraction of the first nonzero layer of $P'$ in $\bfR_{\bfQ_p}$. For each integer $0\leq j\leq e$, multiplication by $\varpi_F^j$ induces an isomorphism of $\bfF_p\pparen{T}$-modules
$$\Lambda_j\colon\bfF_p\pparen{T}\cong\bfR_{\bfQ_p}/\varpi_F\bfR_{\bfQ_p} \xrightarrow{\ \sim\ } \varpi_F^j\bfR_{\bfQ_p}/\varpi_F^{j+1}\bfR_{\bfQ_p}.$$
If $X\in\varpi_F^j\bfR_{\bfQ_p}$, we write
$$\overline{X/\varpi_F^j}\coloneqq\Lambda_j^{-1}\left(X+\varpi_F^{j+1}\bfR_{\bfQ_p}\right)\in\bfF_p\pparen{T}.$$
Using this notation, we extract the leading layer of $P'$ as follows.

\begin{lemma}\label{lem:67522}
	Let $P\in\bfR_{\bfQ_p}$ satisfy $\overline P=T^p$. Then there is a unique integer $s=s(P)\in\{1,\ldots,e\}$ such that $P'\in\varpi_F^s\bfR_{\bfQ_p}\setminus\varpi_F^{s+1}\bfR_{\bfQ_p}$.
	If we put $h_P\coloneqq\overline{P'/\varpi_F^s}\in\bfF_p\pparen{T}^\times$ and $c\coloneqq\overline{p/\varpi_F^e}\in\bfF_p^\times$, then one has the following properties.
	\begin{lemenum}
		\item\label{lem:53419} If $G\in\bfR_{\bfQ_p}$ satisfies $G\circ P=P\circ G$ and $f\coloneqq\overline G\in T\bfF_p\bbrac{T}$ has nonzero linear coefficient, then
		\begin{equation}\label{eq:87025}
			h_P\circ f=(f')^{p-1}h_P.
		\end{equation}
		\item\label{lem:84891} If $H\in T(\calO_F/\varpi_F^{e+1}\calO_F)\bbrac{T}$ has invertible linear coefficient and $P_H\coloneqq H\circ P\circ H\cinv$, then $\overline{P_H}=T^p$, $s(P_H)=s(P)$, and
		\begin{equation}\label{eq:81313}
			h_{P_H}\circ\overline H=(\overline H')^{p-1}h_P.
		\end{equation}
		In particular, the valuation $v_T(h_{P_H})=v_T(h_P)$ is independent of the coordinate.
		\item\label{lem:51094}
		\begin{enumerate}[label=(\alph*)]
			\item If $s<e$, then $h_P$ is the derivative of an element of $\bfF_p\pparen{T}$.
			\item If $s=e$, then $h_P-c T^{p-1}$ is the derivative of an element of $\bfF_p\pparen{T}$.
			      Moreover, one has $h_P'\neq0$ and $v_T(h_P')\leq p-2$.
		\end{enumerate}
	\end{lemenum}
\end{lemma}

\begin{proof}
	Write $P=\sum_{i\in\bfZ}a_i\cdot T^i$ with $a_i\in\calO_F/\varpi_F^{e+1}\calO_F$ for every $i\in \bfZ$. Since $\overline P=T^p$, we have $\overline{a_p}=1$ and $\overline{a_i}=0$ for $i\neq p$. In the formal derivative $P'=\sum_{i\in\bfZ}ia_i\cdot T^{i-1}$, every coefficient lies in $\varpi_F(\calO_F/\varpi_F^{e+1}\calO_F)$, while the coefficient $pa_p$ of $T^{p-1}$ has $\varpi_F$-adic valuation exactly $e=v_{\varpi_F}(p)$ because $a_p$ is a unit. Thus $P'\neq 0$ in $\bfR_{\bfQ_p}$,  implying the existence of $s(P)$ and  $h_P=\overline{P'/\varpi_F^s}\neq 0$.

	To prove \ref{lem:53419}, differentiate $G\circ P=P\circ G$ using the chain rule to obtain
	$$(G'\circ P)P'=(P'\circ G)G'.$$
	Both sides lie in $\varpi_F^s\bfR_{\bfQ_p}$. Passing to $\varpi_F^s \bfR_{\bfQ_p}/\varpi_F^{s+1}\bfR_{\bfQ_p}$ and applying $\Lambda_s^{-1}$ gives
	$$(f'\circ T^p)h_P=(h_P\circ f)f'.$$
	Since $f'\circ T^p=(f')^p$ in characteristic $p$ and $f'\in\bfF_p\bbrac{T}^\times$, \eqref{eq:87025} follows from dividing both sides by $f'$.

	For \Cref{lem:84891}, differentiating $P_H\circ H=H\circ P$ gives $(P_H'\circ H)H'=(H'\circ P)P'$. Since $H'$ and $H'\circ P$ are units in $\bfR_{\bfQ_p}$ and substitution by $H$ is an $(\calO_F/\varpi_F^{e+1}\calO_F)$-linear automorphism, $P_H'$ and $P'$ lie in the exact same $\varpi_F$-adic layer. Hence $s(P_H)=s(P)$, and reducing the $\varpi_F^s$-layer gives \eqref{eq:81313}. Because $\overline H\in T\bfF_p\bbrac{T}$ has nonzero linear coefficient, $\overline H'\in\bfF_p\bbrac{T}^\times$ has valuation $0$ and substitution by $\overline H$ preserves $v_T$, so $v_T(h_{P_H})=v_T(h_P)$.

	For the last assertion, note that $h_P$ is the derivative of an element $A$ of $\bfF_p\pparen{T}$ if and only if its coefficients in all degrees congruent to $-1$ modulo $p$ vanish. If we write $P'=\sum_{i\in\bfZ}ia_i\cdot T^{i-1}$,  then for every integer $l$, the coefficient of $T^{lp-1}$ is $lp\cdot a_{lp}$, which has $\varpi_F$-adic valuation at least $e$. When $s<e$, this shows that $h_P$ has no term of degree $lp-1$ for any $l\in\bfZ$, and proves \Cref{lem:51094} in this case.

	When $s=e$, the contribution of $lp\cdot a_{lp}$ to $h_P$ is $lc\cdot\overline{a_{lp}}$, which equals $c$ for $l=1$ and $0$ for $l\neq 1$ as $\overline{a_{lp}}=0$. Separating the $l=1$ term from the remaining terms gives $h_P=cT^{p-1}+A'$ for some $A\in\bfF_p\pparen{T}$. Differentiating once more gives $h_P'=-cT^{p-2}+A''$. Since $A''$ has no term of degree $p-2\equiv -2\pmod p$, the coefficient of $T^{p-2}$ in $h_P'$ is $-c\neq 0$, so $h_P'\neq 0$ and $v_T(h_P')\leq p-2$.
\end{proof}

\begin{definition}\label{def:51837}
	For a lift $P\in\bfR_{\bfQ_p}$ with $\overline P=T^p$, we call $s(P)\in\{1,\ldots,e\}$ the \textbf{derivative level} of $P$ and $h_P\in\bfF_p\pparen{T}^\times$ the \textbf{normalized derivative} of $P$.
\end{definition}

\section{Case of $s(P)=e$: normal forms and valuation bounds}\label{sec:61843}

In this section, we study the case $s(P)=e$. Our goal is to bound $v_T(h_P)$ from above in terms of the ramification of a commuting automorphism $\sigma\in\Gamma$, which is achieved by conjugating $f_{\sigma,\pi}$ into the \textbf{normal form} $T+\alpha T^{d+1}+O(T^{2d+1})$ ($\alpha\in\bfF_p^\times$).

\subsection{The valuation bound at derivative level \texorpdfstring{$e$}{e}}\label{subsec:53178}

From now on, we write $O(T^n)$ (resp.\ $O(\pi^n)$) for an unspecified element of $\bfF_p\pparen{T}$ (resp.\ $\bfE_{\bfQ_p}$) of valuation at least $n$.

\begin{proposition}\label{thm:10556}
	Let $P,G\in\bfR_{\bfQ_p}$ satisfy $\overline P=T^p$, $G\circ P=P\circ G$, and $s(P)=e$. If the reduction $f\coloneqq\overline G$ has the form
	$$f=T+\alpha T^{d+1}+O\left(T^{2d+1}\right),\ \alpha\in\bfF_p^\times,\  d\geq1,\  d\equiv-1\pmod p,$$
	then the normalized derivative $h_P$ satisfies
	$$v_T(h_P)\leq p-1-d.$$
\end{proposition}

\begin{proof}
	Put $m\coloneqq v_T(h_P)$. We are going to evaluate the leading term of the difference $h_P\circ f-h_P$ in two ways.%

	First, since $d\equiv -1\pmod p$, the term $\alpha T^{d+1}$ vanishes after the differentiation. As a result, we have $f'=1+O(T^{2d})$ and consequently $(f')^{p-1}=1+O(T^{2d})$. The identity $h_P\circ f=(f')^{p-1}h_P$ in \Cref{lem:53419} implies
	\begin{equation}\label{eq:27200}
		h_P\circ f-h_P=\bigl((f')^{p-1}-1\bigr)h_P=O\left(T^{m+2d}\right).
	\end{equation}

	On the other hand, if we write $h_P=\sum_{n\geq m}h_nT^n$, then the binomial expansion gives
	\begin{align*}
		h_P\circ f-h_P= & \sum_{n\geq m}h_nT^n\left(1+\alpha T^d+O\left(T^{2d}\right)\right)^n-h_P                                  \\
		=               & \sum_{n\geq m}h_nT^n\left(1+n\left(\alpha T^d+O\left(T^{2d}\right)\right)+O\left(T^{2d}\right)\right)-h_P \\
		=               & T\left(\alpha T^d+O\left(T^{2d}\right)\right)h_P'+O\left(T^{m+2d}\right).
	\end{align*}
	Since $v_T(h_P')\geq m-1$, we know that
	$$h_P\circ f-h_P=\alpha T^{d+1}h_P'+O\left(T^{m+2d}\right).$$
	Comparing this expansion with equality \eqref{eq:27200} shows that $v_T\left(\alpha T^{d+1}h_P'\right)\geq m+2d$, i.e.
	$$d+1+v_T(h_P')\geq m+2d.$$

	On the other hand, since $s(P)=e$, \Cref{lem:51094} guarantees that $h_P'\neq 0$ and $v_T(h_P')\leq p-2$. Combining the two inequalities gives
	$$m+2d\leq d+1+v_T(h_P')\leq d+p-1,$$
	which implies the desired inequality.
\end{proof}

\subsection{Eliminating intermediate terms via normal forms}\label{subsec:21426}

For any $g\in\Gamma$, define its \textbf{ramification number} on $\bfE_{\bfQ_p}$ by
$$i_{\bfE_{\bfQ_p}}(g)\coloneqq v_{\bfE}(g\ast\pi-\pi)-1=v_T(f_{g,\pi}-T)-1,$$
with $i_{\bfE_{\bfQ_p}}(g)=+\infty$ if $g$ acts trivially; this is independent of the choice of uniformizer $\pi$. To conjugate an automorphism into normal form, we first establish two lemmas on the operator $\sigma-\id$.

\begin{lemma}\label{lem:32526}
	Suppose that $\sigma\in\Gamma$ and $\pi$ is a uniformizer of $\bfE_{\bfQ_p}$ such that
	$$\sigma\ast\pi=\pi+\alpha\pi^{d+1}+O(\pi^{d+2}),\  \alpha\in\bfF_p^\times,\  d\in\bfZ_{\geq 1}.$$
	Then $v_{\bfE}(\sigma\ast x-x)\geq v_{\bfE}(x)+d$ for every $x\in\bfE_{\bfQ_p}$. More precisely, if $x=b\pi^n+O\left(\pi^{n+1}\right)$ with $b\in\bfF_p$ and $n\in\bfZ$, then
	$$\sigma\ast x-x=nb\alpha\pi^{n+d}+O\left(\pi^{n+d+1}\right).$$
\end{lemma}

\begin{proof}
	Write $\sigma\ast\pi=\pi(1+y)$ with $y=\alpha\pi^d+O(\pi^{d+1})$. For every integer $n\in\bfZ$, the binomial expansion gives $(1+y)^n=1+ny+O(y^2)$. Since $2d\geq d+1$, we obtain $\sigma\ast\pi^n-\pi^n=n\alpha\pi^{n+d}+O(\pi^{n+d+1})$. Applying this formula term by term to the Laurent expansion of $x$ yields both assertions.
\end{proof}

\begin{lemma}\label{lem:45124}
	Suppose that $\sigma\in\Gamma$ and $\pi$ is a uniformizer of $\bfE_{\bfQ_p}$ such that
	$$\sigma\ast\pi=\pi+\alpha\pi^{d+1}+\beta\pi^\ell+O\left(\pi^{\ell+1}\right),$$
	where $\alpha,\beta\in\bfF_p^\times$, $d+2\leq\ell\leq2d$, and $d\equiv\ell\equiv-1\pmod p$. Then
	$$i_{\bfE_{\bfQ_p}}(\sigma^p)=\ell+(p-1)d-1<(p+1)d.$$
\end{lemma}

\begin{proof}
	Let $N\coloneqq\sigma-\id$. Then one has $N^p=\sigma^p-\id$, $i_{\bfE_{\bfQ_p}}(\sigma^p)=v_{\bfE}(N^p(\pi))-1$ and
	$$N(\pi)=\alpha\pi^{d+1}+\beta\pi^\ell+O\left(\pi^{\ell+1}\right).$$
	We apply $N^{p-1}$ to each of these terms.

	For the term $\alpha\pi^{d+1}$, since $p\mid d+1$, we know that $N\left(\pi^{d+1}\right)=N\left(\pi^{(d+1)/p}\right)^p$. By \Cref{lem:32526}, one has $$v_{\bfE}\left(N\left(\pi^{(d+1)/p}\right)\right)\geq \frac{d+1}{p}+d,$$
	so $v_{\bfE}\left(N(\pi^{d+1})\right)\geq d+1+pd$. Applying $N^{p-2}$ then raises the valuation by at least $(p-2)d$, giving
	$$v_{\bfE}\left(N^{p-1}\left(\alpha\pi^{d+1}\right)\right)\geq (2p-1)d+1>2d+(p-1)d\geq\ell+(p-1)d.$$

	For the term $\beta\pi^\ell$, applying \Cref{lem:32526} repeatedly gives
	$$N^{p-1}(\beta\pi^\ell)=\beta\alpha^{p-1}\left(\prod_{r=0}^{p-2}(\ell+rd)\right)\pi^{\ell+(p-1)d}+O\left(\pi^{\ell+(p-1)d+1}\right).$$
	Since $\ell\equiv d\equiv -1\pmod p$, we have $\ell+rd\equiv -(r+1)\not\equiv 0\pmod p$ for every $0\leq r\leq p-2$. Hence $\prod_{r=0}^{p-2}(\ell+rd)$ is nonzero in $\bfF_p$, and $v_{\bfE}(N^{p-1}(\beta\pi^\ell))=\ell+(p-1)d$.

	For the remainder $O\left(\pi^{\ell+1}\right)$, since $N$ increases valuation by at least $d$ at each step by \Cref{lem:32526}, we have $v_{\bfE}\left(N^{p-1}\left(O\left(\pi^{\ell+1}\right)\right)\right)\geq \ell+1+(p-1)d$.

	Combining these estimates with the ultrametric inequality gives the desired equality.
\end{proof}

Now we prove the existence of a uniformizer in which $f_{\sigma,\pi}$ takes the normal form.
\begin{proposition}\label{thm:41759}
	Let $\sigma$ be an element of $\Gamma$ satisfying the following conditions:
	\begin{enumerate}
		\item The integer $d\coloneqq i_{\bfE_{\bfQ_p}}(\sigma)$ satisfies $d\equiv-1\pmod p$.
		\item $i_{\bfE_{\bfQ_p}}(\sigma^p)\geq(p+1)d$.
	\end{enumerate}
	Then there exists a uniformizer $\pi$ of $\bfE_{\bfQ_p}$ and an element $\alpha\in\bfF_p^\times$ such that
	\begin{equation}\label{eq:82881}
		\sigma\ast\pi=\pi+\alpha\pi^{d+1}+O\left(\pi^{2d+1}\right),
	\end{equation}
	or equivalently, $f_{\sigma,\pi}=T+\alpha T^{d+1}+O\left(T^{2d+1}\right)$.
\end{proposition}

\begin{proof}
	Pick an arbitrary uniformizer $\varpi$ of $\bfE_{\bfQ_p}$, and write
	$$\sigma\ast\varpi=\varpi+\alpha\varpi^{d+1}+\beta\varpi^\ell+O\left(\varpi^{\ell+1}\right),\  \alpha\in\bfF_p^\times,\ \beta\in\bfF_p.$$
	If $\ell\geq 2d+1$ then we are done. Otherwise, we assume that $\ell\in\{d+2,\ldots,2d\}$ and $\beta\neq 0$. Note that under the hypotheses of the proposition, \Cref{lem:45124} implies that $\ell\not\equiv-1\pmod p$, and consequently $\ell-2d-1\equiv\ell+1\not\equiv0\pmod p$. This allows us to write $t\coloneqq -\frac{\beta}{(\ell-2d-1)\alpha}\in\bfF_p^\times$ and set $\varpi'\coloneqq\varpi+t\varpi^{\ell-d}$. We are going to prove that
	\begin{equation*}
		\sigma\ast\varpi'=\varpi'+\alpha(\varpi')^{d+1}+O\left((\varpi')^{\ell+1}\right),
	\end{equation*}
	and the desired uniformizer can be obtained by repeating this procedure. In fact, we have
	\begin{equation}\label{eq:57688}
		\sigma\ast\varpi'-\varpi'-\alpha(\varpi')^{d+1}=(\sigma\ast\varpi-\varpi)+t\left(\sigma\ast\left(\varpi^{\ell-d}\right)-\left(\varpi^{\ell-d}\right)\right)-\alpha\varpi^{d+1}\left(1+t\varpi^{\ell-d-1}\right)^{d+1}.
	\end{equation}
	\Cref{lem:32526} gives
	$$\sigma\ast\left(\varpi^{\ell-d}\right)-\left(\varpi^{\ell-d}\right)=(\ell-d)\alpha\varpi^\ell+O\left(\varpi^{\ell+1}\right),$$
	and the binomial expansion implies
	$$\left(1+t\varpi^{\ell-d-1}\right)^{d+1}=1+(d+1)t \varpi^{\ell-d-1}+O\left(\varpi^{\ell-d}\right).$$
	Substituting these expansions into \eqref{eq:57688}, we obtain
	\begin{align*}
		\sigma\ast\varpi'-\varpi'-\alpha(\varpi')^{d+1}= & \alpha\varpi^{d+1}+\beta\varpi^\ell+O\left(\varpi^{\ell+1}\right)+t\left((\ell-d)\alpha\varpi^\ell+O\left(\varpi^{\ell+1}\right)\right) \\
		                                                 & \quad-\alpha\varpi^{d+1}\left(1+(d+1)t \varpi^{\ell-d-1}+O\left(\varpi^{\ell-d}\right)\right)                                           \\
		=                                                & O\left(\varpi^{\ell+1}\right).
	\end{align*}
	Since $\varpi$ and $\varpi'$ have the same valuation, the result follows.
\end{proof}

\section{Case of $s(P)<e$: invariant exact differentials}\label{sec:42840}

When $s(P)<e$, \Cref{lem:51094} shows that $h_P=A'$ for some $A\in\bfF_p\pparen{T}$, so the upper bound on $v_T(h_P')$ used in \Cref{sec:61843} is no longer available. Instead, we use the exactness of $h_P\d T$ together with the tame generator $\delta$ to construct a $\tau$-invariant exact differential on $\bfF_p\pparen{T}$, and show that the existence of such a differential imposes an integrality constraint on the $p$-adic limit of ramification numbers.

Throughout this section, we fix the uniformizer $u$ of $\bfE_{\bfQ_p}$ from \Cref{subsec:13146}, abbreviate $f_g\coloneqq f_{g,u}$ for $g\in\Gamma$, and let $\Gamma$ act on $\bfF_p\pparen{T}$ via $g(x)\coloneqq x\circ f_g$.

\subsection{From lower-level lifts to invariant exact differentials}\label{subsec:86873}

Let $\bfF_p\pparen{T}\d T$ be the space of continuous differentials on $\bfF_p\pparen{T}$, equipped with the exterior derivative $\d A=A'\d T$ and the action of $\sigma\in\Gamma$ given by
$$\sigma(a\d T)\coloneqq (a\circ f_\sigma)f_\sigma'\d T.$$
For a nonzero differential $\omega=a\d T$, we set $\opn{ord}_T(\omega)\coloneqq v_T(a)$. %

\begin{lemma}\label{lem:45874}
	Let $P,F_\tau,F_\delta\in\bfR_{\bfQ_p}$ satisfy $\overline P=T^p$, $\overline{F_\tau}=f_\tau$, $\overline{F_\delta}=f_\delta$, $F_\tau\circ P=P\circ F_\tau$, and $F_\delta\circ P=P\circ F_\delta$. Suppose that $f_\tau\in T+T^2\bfF_p\bbrac{T}$ and that the linear coefficient $\lambda\in\bfF_p^\times$ of $f_\delta=\lambda T+O(T^2)$ has multiplicative order $p-1$. If $s(P)<e$, then there exists a nonzero exact differential $\omega\in\bfF_p\pparen{T}\d T$ such that $\tau(\omega)=\omega$.
\end{lemma}

\begin{proof}
	Set $m\coloneqq v_T(h_P)$. Since $s(P)<e$, \Cref{lem:67522} provides $A\in\bfF_p\pparen{T}$ such that $h_P=A'$, together with the two functional equations
	\begin{equation}\label{eq:28265}
		h_P\circ f_\tau=(f_\tau')^{p-1}h_P,\
		h_P\circ f_\delta=(f_\delta')^{p-1}h_P.
	\end{equation}
	Write $h_P=bT^m(1+z)$ with $b\in\bfF_p^\times$ and $z\in T\bfF_p\bbrac{T}$. Comparing the coefficients of degree $m$ in the second identity of \eqref{eq:28265} gives $b\lambda^m=b\lambda^{p-1}=b$. Since $b\neq 0$ and $\lambda\in\bfF_p^\times$ has exact order $p-1$, we deduce that $p-1\mid m$.

	Because $p-1$ is invertible in $\bfF_p$,
	Hensel's lemma guarantees that the unit $(1+z)^{-1}\in 1+T\bfF_p\bbrac{T}$ admits a unique $(p-1)$-st root in $1+T\bfF_p\bbrac{T}$. Combined with $p-1\mid m$, this allows us to define $a\in T^{-m/(p-1)}(1+T\bfF_p\bbrac{T})\subset\bfF_p\pparen{T}^\times$ such that
	$$a^{p-1}=\frac{b}{h_P}.$$
	Taking the $(p-1)$-st power of $\tau(a\d T)/(a\d T)$ and using the first identity in \eqref{eq:28265} gives
	$$\left(\frac{(a\circ f_\tau)f_\tau'}{a}\right)^{p-1}=\frac{h_P}{h_P\circ f_\tau}(f_\tau')^{p-1}=1.$$
	Since $f_\tau\in T+T^2\bfF_p\bbrac{T}$, the quotient $(a\circ f_\tau)f_\tau'/a$ lies in $1+T\bfF_p\bbrac{T}$. The only $(p-1)$-st root of unity in $1+T\bfF_p\bbrac{T}$ is $1$, so $(a\circ f_\tau)f_\tau'=a$. Thus $\omega\coloneqq a\d T$ is a nonzero differential fixed by $\tau$.

	Finally, since $a^{p-1}h_P=b$, multiplying by $a$ gives $a^ph_P=ba$. Using $h_P=A'$ and $(a^p)'=0$ in characteristic $p$, we find
	$$\omega=a\d T=b^{-1}a^ph_P\d T=b^{-1}a^pA'\d T=\d\left(b^{-1}a^pA\right),$$
	so $\omega$ is exact.
\end{proof}

\subsection{Approximate invariants and the \texorpdfstring{$p$}{p}-adic limit of ramification numbers}\label{subsec:20375}

To rule out $\tau$-invariant exact differentials, we first establish a general property of sequences of \emph{approximate invariants} of an infinite-order automorphism $\sigma\in\Gamma$, namely sequences $(x_n)_{n\geq 0}$ in $\bfF_p\pparen{T}^\times$ satisfying $v_T(\sigma(x_n)-x_n)-v_T(x_n)\to+\infty$.

\begin{lemma}\label{lem:44677}
	Let $\sigma\in\Gamma$ act with infinite order on $\bfE_{\bfQ_p}$, and suppose $f_\sigma\in T+T^2\bfF_p\bbrac{T}$. Let $(x_n)_{n\geq 0}$ be a sequence in $\bfF_p\pparen{T}^{\times}$ satisfying
	$$\lim_{n\to\infty}v_T\left(\frac{\sigma(x_n)}{x_n}-1\right)=+\infty$$
	If the sequence $\{v_T(x_n)\}_{n\geq 1}$ converges $p$-adically to $m\in\bfZ_p$, then $m=0$.
\end{lemma}

\begin{proof}
	Put $f\coloneqq f_\sigma$, and write each $x_n$ as $x_n=T^{m_n}U_n$ with $m_n=v_T(x_n)\in\bfZ$ and $U_n\in\bfF_p\bbrac{T}^{\times}$. Multiplying $x_n$ by $U_n(0)^{-1}\in\bfF_p^\times$ does not change $v_T(x_n)$ or $v_T\left(\frac{\sigma(x_n)}{x_n}-1\right)$, so we may assume $U_n\in 1+T\bfF_p\bbrac{T}$. By compactness of $1+T\bfF_p\bbrac{T}$ with respect to the $T$-adic metric, we may pass to a subsequence so that $U_n$ converges to some $U\in 1+T\bfF_p\bbrac{T}$. %
	For any $z\in1+T\bfF_p\bbrac{T}$ and $c\in\bfZ_p$, the power $z^c\in 1+T\bfF_p\bbrac{T}$ is well defined by continuity from integer exponents. %
	Since $m_n\lto m$ in $\bfZ_p$ and $f/T\in 1+T\bfF_p\bbrac{T}$, passing to the limit in
	$$\frac{\sigma(x_n)}{x_n}=\left(\frac fT\right)^{m_n}\frac{U_n\circ f}{U_n}\longrightarrow 1$$
	yields the functional equation
	\begin{equation}\label{eq:66693}
		\left(\frac fT\right)^m(U\circ f)=U.
	\end{equation}

	For every $j\in\bfN$, we set $f_j\coloneqq f_{\sigma^{p^j}}=f^{\circ p^j}$ and $d_j\coloneqq v_T(f_j-T)-1$. Since $\sigma$ has infinite order and $f\in T+T^2\bfF_p\bbrac{T}$, $\{d_j\}_{j\geq 1}$ is a strictly increasing sequence of positive integers. %
	Composing \eqref{eq:66693} with $f^{\circ i}$ gives
	$$\left(\frac{f^{\circ i+1}}{f^{\circ i}}\right)^m\left(U\circ f^{\circ i+1}\right)=U\circ f^{\circ i},\qquad i=0,1,\ldots,p^j-1.$$
	Multiplying these $p^j$ identities and cancelling the intermediate factors gives
	\begin{equation}\label{eq:44455}
		U\circ f_j-U=U\left(\left(\frac{T}{f_j}\right)^m-1\right).
	\end{equation}

	Assume for contradiction that $m\neq 0$, and set $r\coloneqq v_p(m)\geq 0$. For any $z\in 1+T\bfF_p\bbrac{T}\setminus\{1\}$, the binomial expansion of $z^{p^r}$ shows that $v_T(z^m-1)=p^r v_T(z-1)$. Since $v_T(T/f_j-1)=d_j$, the right-hand side of equation \eqref{eq:44455} has $T$-adic valuation exactly $p^rd_j$.

	If $r=0$, the right-hand side of \eqref{eq:44455} has valuation $d_j$, whereas applying \Cref{lem:32526} to $U-U(0)$ shows that the left-hand side has valuation $v_T(U\circ f_j-U)\geq d_j+1$, a contradiction.

	If $r\geq 1$, we first claim that $U'=0$. Indeed, if $U'\neq 0$, writing $f_j=T+\delta_j$ with $v_T(\delta_j)=d_j+1$ gives the Taylor expansion $U\circ f_j-U=\delta_jU'+O\left(T^{2d_j+2}\right)$. For sufficiently large $j$, we have $2d_j+2>d_j+1+v_T(U')$, so the left-hand side of \eqref{eq:44455} has valuation $d_j+1+v_T(U')$. This cannot equal $p^rd_j$ for all large $j$ because $p^r\geq p\geq 3$ and $d_j\lto+\infty$. Thus $U'=0$, so $U=V^p$ for a unique element $V\in 1+T\bfF_p\bbrac{T}$. Taking $p$-th roots in \eqref{eq:66693} then gives
	$$\left(\frac fT\right)^{m/p}(V\circ f)=V,$$
	which is the same identity with $(U,m)$ replaced by $(V,m/p)$ and $v_p(m/p)=r-1$. Repeating this descent will eventually reach the case $r=0$, which we have already shown to be impossible. The result follows.
\end{proof}

We now apply \Cref{lem:44677} to the orbit sums constructed from an invariant exact differential.

\begin{proposition}\label{prop:15106}
	Let $\sigma\in\Gamma$ act with infinite order on $\bfE_{\bfQ_p}$ and satisfy $f_\sigma\in T+T^2\bfF_p\bbrac{T}$. Suppose that the ramification numbers
	$$d_n\coloneqq i_{\bfE_{\bfQ_p}}\left(\sigma^{p^n}\right)=v_T(f_{\sigma^{p^n}}-T)-1$$
	converge $p$-adically to $b\in\bfZ_p$ as $n\lto\infty$. If $\sigma$ fixes a nonzero exact differential $\omega\in\bfF_p\pparen{T}\d T$, then
	$$b=-\opn{ord}_T(\omega)-1\in\bfZ.$$
\end{proposition}

\begin{proof}
	Since $\omega$ is exact, we may choose $U\in\bfF_p\pparen{T}$ with $\d U=U'\d T=\omega$. Deleting all terms of $U$ whose exponents are divisible by $p$ does not change $U'$, so we may assume that $U$ has no terms in degrees divisible by $p$. In particular, its leading exponent $q\coloneqq v_T(U)\in\bfZ$ satisfies
	$$q=\opn{ord}_T(\omega)+1,\  p\nmid q.$$
	Because $\sigma(\omega)=\omega$, we have $\d(\sigma(U)-U)=\sigma(\d U)-\d U=0$. In characteristic $p$, the kernel of $\d$ is $\bfF_p\pparen{T}^p$, so there is a unique $B\in\bfF_p\pparen{T}$ such that $\sigma(U)-U=B^p$. For each $n\geq 0$, define the orbit sum
	\begin{equation}\label{eq:13709}
		x_n\coloneqq\sum_{j=0}^{p^n-1}\sigma^j(B)\in\bfF_p\pparen{T}.
	\end{equation}
	Raising $x_n$ to the $p$-th power, one obtains
	\begin{equation}\label{eq:94654}
		x_n^p=\sum_{j=0}^{p^n-1}\sigma^j(B^p)=\sum_{j=0}^{p^n-1}\bigl(\sigma^{j+1}(U)-\sigma^j(U)\bigr)=\sigma^{p^n}(U)-U.
	\end{equation}
	Since $p\nmid q=v_T(U)$, applying \Cref{lem:32526} to $\sigma^{p^n}$ shows that $\sigma^{p^n}(U)-U$ is nonzero of valuation $q+d_n$. By \eqref{eq:94654}, it follows that $B\neq 0$, $x_n\neq 0$, and
	$$v_T(x_n)=\frac{q+d_n}{p}\in\bfZ.$$
	On the other hand, applying $\sigma-\id$ to \eqref{eq:13709} gives $\sigma(x_n)-x_n=\sigma^{p^n}(B)-B$. By \Cref{lem:32526}, we have $v_T\left(\sigma^{p^n}(B)-B\right)\geq v_T(B)+d_n$, and hence
	$$v_T\left(\frac{\sigma(x_n)}{x_n}-1\right)\geq v_T(B)+d_n-\frac{q+d_n}{p}=v_T(B)-\frac{q}{p}+\left(1-\frac{1}{p}\right)d_n\longrightarrow+\infty$$
	as $n\lto\infty$. Finally, since $q+d_n=pv_T(x_n)\in p\bfZ$ and $d_n\lto b$ in $\bfZ_p$, the integers $v_T(x_n)=(q+d_n)/p$ converge $p$-adically to $(q+b)/p\in\bfZ_p$. Applying \Cref{lem:44677} to $(x_n)_{n\geq 0}$ gives $(q+b)/p=0$, i.e. $b=-q=-\opn{ord}_T(\omega)-1\in\bfZ$.
\end{proof}

\section{Arithmetic inputs from the false Tate curve extension}\label{sec:93558}

In this section, we establish the arithmetic properties of the false Tate curve extension $\bfQ_p^{\FT}/\bfQ_p$ needed for the obstructions of \Cref{sec:61843} and \Cref{sec:42840}: the ramification numbers of the Kummer iterates $\sigma_k\coloneqq\tau^{p^k}$ and the leading term of the tame generator $\delta$.

We first compute $d_k\coloneqq i_{\bfE_{\bfQ_p}}(\sigma_k)$ from Viviani's explicit description of the ramification filtration on the finite layers
$$L_n\coloneqq\bfQ_p\left(\zeta_{p^n},p^{1/p^n}\right),\  \Gamma_n\coloneqq\Gal(L_n/\bfQ_p).$$
For $g\in\Gamma_n$ with $g\neq\id$, set $i_{L_n}(g)\coloneqq v_{L_n}(g(\varpi)-\varpi)-1$, where $v_{L_n}$ is the normalized valuation on $L_n$ and $\varpi$ is any uniformizer of $L_n$. Set $i_{L_n}(\id)\coloneqq+\infty$. Let $(\Gamma_n)_s\coloneqq\{g\in\Gamma_n:i_{L_n}(g)\geq s\}$ be the lower numbering ramification group, and let $\Gamma_n^t\coloneqq(\Gamma_n)_{\psi_n(t)}$ be the upper numbering ramification group, where $\psi_n$ is the Herbrand function of $L_n/\bfQ_p$.

\begin{proposition}\label{prop:78229}
	For every $k\geq0$, the ramification number of $\sigma_k=\tau^{p^k}$ on $\bfE_{\bfQ_p}$ is given by
	\begin{equation}\label{eq:16300}
		d_k\coloneqq i_{\bfE_{\bfQ_p}}(\sigma_k)=\frac{2p^{2k+2}+p-1}{p+1}.
	\end{equation}
	In particular, one has
	\begin{propenum}
		\item $d_0=2p-1$, $d_k\equiv-1\pmod p$, and $d_{k+1}=p^2d_k-(p-1)^2\geq(p+1)d_k$ for all $k\geq 0$.
		\item As $k\to\infty$, one has $d_k\lto+\infty$ in $\bfR$, and
		\begin{equation}\label{eq:56188}
			d_k\xlongrightarrow{p\text{-adically}}\frac{p-1}{p+1}\in\bfZ_p\backslash\bfZ.
		\end{equation}
	\end{propenum}
\end{proposition}
\begin{proof}
	\cite[Remark 4.5]{caisCharacterizationStrictlyAPF2016} shows that the extension $\bfQ_p^{\FT}/\bfQ_p$ is a strictly arithmetically profinite extension. Since the finite Galois extensions $L_n$ form a cofinal tower in $\bfQ_p^{\FT}$, Fontaine and Wintenberger's relation between the ramification filtrations of $L_n/\bfQ_p$ and of $\bfE_{\bfQ_p}$ (cf. \cite[Proposition 3.3.2]{wintenbergerCorpsNormesCertaines1983}) implies that for every $g\in\Gamma$,
	\begin{equation}\label{eq:52480}
		i_{\bfE_{\bfQ_p}}(g)=i_{L_n}(g|_{L_n})\ \text{for all sufficiently large }n.
	\end{equation}

	Fix $k\geq 0$ and choose $n\geq k+3$ large enough for \eqref{eq:52480} to hold for $g=\sigma_k$. We read off $i_{L_n}(\sigma_k|_{L_n})$ from the upper ramification filtration of $\Gamma_n$ computed in \cite[Corollary~6.10]{vivianiRamificationGroupsArtin2004}. In the notation of \cite{vivianiRamificationGroupsArtin2004}, one has
	$$\Gamma_n=C(p^n)\rtimes G(p^n),\  G(p^n)\cong(\bfZ/p^n\bfZ)^\times,$$
	where $C(p^n)$ is the cyclic Kummer subgroup of order $p^n$ generated by $\tau|_{L_n}$, $C(p^a)$ denotes its unique subgroup of order $p^a$, and $G(p^n)^a$ denotes the subgroup of $(\bfZ/p^n\bfZ)^\times$ consisting of classes congruent to $1$ modulo $p^a$.

	Set $t_k\coloneqq k+1+\frac{1}{p-1}$. By \cite[Corollary~6.10]{vivianiRamificationGroupsArtin2004}, the upper ramification groups at and immediately after $t_k$ are
	$$\Gamma_n^{t_k}=C(p^{n-k})\rtimes G(p^n)^{k+2},\ \Gamma_n^t=C(p^{n-k-1})\rtimes G(p^n)^{k+2}\ (t_k<t\leq k+2),$$
	and the lower break corresponding to the upper break $t_k$ is
	\begin{equation}\label{eq:54219}
		\psi_n(t_k)=\frac{2p^{2k+2}+p-1}{p+1}.
	\end{equation}
	Since $\sigma_k=\tau^{p^k}$, its restriction $\sigma_k|_{L_n}$ is a generator of $C(p^{n-k})$. As a result, $\sigma_k|_{L_n}$ belongs to $\Gamma_n^{t_k}=(\Gamma_n)_{\psi_n(t_k)}$ but not to $\Gamma_n^t$ for any $t>t_k$. Hence $i_{L_n}(\sigma_k|_{L_n})=\psi_n(t_k)$, and equality \eqref{eq:16300} follows from \eqref{eq:52480} and \Cref{eq:54219}.

	The remaining assertions follow from \eqref{eq:16300} by straightforward calculations.
\end{proof}

\begin{remark}\label{rem:24777}
	In \cite[Corollary~6.10]{vivianiRamificationGroupsArtin2004}, the lower index attached to the integer upper break
	$$\Gamma_n^{i+1}=C(p^{n-i})\rtimes G(p^n)^{i+1}\ (1\leq i\leq n-2)$$
	is printed as $(p-1)(p^{2i}+1)/(p+1)$, which is a misprint for $(p-1)(p^{2i+1}+1)/(p+1)$. Indeed, on the interval $i+\frac{1}{p-1}<t\leq i+1$ of length $\frac{p-2}{p-1}$, the subgroup $\Gamma_n^t$ has index $[\Gamma_n:\Gamma_n^t]=(p-1)p^{2i}$, so integrating $[\Gamma_n:\Gamma_n^t]$ from the preceding fractional break $i+\frac{1}{p-1}$ gives
	$$\psi_n(i+1)=\frac{2p^{2i}+p-1}{p+1}+\frac{p-2}{p-1}\cdot(p-1)p^{2i} =\frac{(p-1)(p^{2i+1}+1)}{p+1}.$$
	Notice that the proof of \Cref{prop:78229} only uses the fractional breaks $t_k=k+1+\frac{1}{p-1}$, whose lower indices $\psi_n(t_k)$ are stated correctly in \cite{vivianiRamificationGroupsArtin2004}.
\end{remark}

Next, we use the semi-direct product relation $\delta\tau\delta^{-1}=\tau^{\chi(\delta)}$ from \Cref{lem:23436} together with $d_0=2p-1$ to determine the linear coefficient of the tame generator $\delta$.

\begin{lemma}\label{lem:99761}
	For every uniformizer $\pi$ of $\bfE_{\bfQ_p}$, one has
	$$f_{\delta,\pi}=\overline{\chi(\delta)}\,T+O\left(T^2\right),$$
	where $\chi$ is the $p$-adic cyclotomic character. In particular, the linear coefficient of $f_{\delta,\pi}$ has multiplicative order $p-1$ in $\bfF_p^\times$.
\end{lemma}

\begin{proof}
	Set $d\coloneqq d_0=i_{\bfE_{\bfQ_p}}(\tau)$, so that $d=2p-1$ by \Cref{prop:78229}. By \Cref{lem:16018}, we may write
	$$f_{\tau,\pi}=T+\alpha T^{d+1}+O(T^{d+2}),\  f_{\delta,\pi}=\lambda T+O(T^2),\ \alpha,\lambda\in\bfF_p^\times.$$
	By \Cref{lem:31185}, the relation $\delta\tau\delta^{-1}=\tau^{\chi(\delta)}$ translates into the identity
	\begin{equation}\label{eq:57340}
		f_{\delta,\pi}\cinv\circ f_{\tau,\pi}\circ f_{\delta,\pi}=f_{\tau^{\chi(\delta)},\pi}.
	\end{equation}
	We compare the coefficients of $T^{d+1}$ on both sides of \eqref{eq:57340}. On the left-hand side, substituting $f_{\delta,\pi}=\lambda T+O(T^2)$ into $f_{\tau,\pi}$ gives
	$$f_{\tau,\pi}\circ f_{\delta,\pi}=f_{\delta,\pi}+\alpha\lambda^{d+1}T^{d+1}+O\left(T^{d+2}\right).$$
	Since $f_{\delta,\pi}\cinv=\lambda^{-1}T+O(T^2)$ satisfies $f_{\delta,\pi}\cinv\circ f_{\delta,\pi}=T$ and $(f_{\delta,\pi}\cinv)'\circ f_{\delta,\pi}=\lambda^{-1}+O(T)$, composing on the left with $f_{\delta,\pi}\cinv$ and expanding at $f_{\delta,\pi}$ yields
	$$f_{\delta,\pi}\cinv\circ f_{\tau,\pi}\circ f_{\delta,\pi}=T+\alpha\lambda^dT^{d+1}+O\left(T^{d+2}\right).$$

	On the right-hand side of \eqref{eq:57340}, for every integer $N\geq 1$, iterating $f_{\tau,\pi}$ gives
	$$f_{\tau^N,\pi}=f_{\tau,\pi}^{\circ N}=T+\overline N\alpha T^{d+1}+O(T^{d+2}).$$
	Choosing a sequence of positive integers $N$ converging to $\chi(\delta)$ in $\bfZ_p$. By continuity of the $\Gamma$-action, the coefficient of $T^{d+1}$ in $f_{\tau^{\chi(\delta)},\pi}$ is $\overline{\chi(\delta)}\alpha$.

	Since $\alpha\neq 0$, comparing the degree-$(d+1)$ coefficients gives $\lambda^d=\overline{\chi(\delta)}$. Finally, since $\lambda^{p-1}=1$ and $d=2p-1\equiv 1\pmod{p-1}$, we have $\lambda^d=\lambda$, and hence $\lambda=\overline{\chi(\delta)}$. By definition of $\delta$ as a generator of the torsion subgroup of $\Gal(\bfQ_p^{\FT}/\bfQ_p^{\opn{Kum}})\cong\bfZ_p^\times$, the reduction $\overline{\chi(\delta)}$ is a generator of $\bfF_p^\times$ and thus has order $p-1$.
\end{proof}

\section{Proof of the main results}\label{sec:91137}

We now assemble the results of the preceding sections to prove \Cref{thm:42866} and \Cref{thm:84624}. %

\begin{proof}[Proof of \Cref{thm:42866}]
	Assume for contradiction that there exist $P,F_\tau,F_\delta\in\bfR_{\bfQ_p}$ satisfying
	$$\overline P=T^p,\  \overline{F_\tau}=f_{\tau,u},\  \overline{F_\delta}=f_{\delta,u},\  F_\tau\circ P=P\circ F_\tau,\  F_\delta\circ P=P\circ F_\delta.$$
	Let $s\coloneqq s(P)\in\{1,\ldots,e\}$ be the derivative level of $P$, let $h_P\in\bfF_p\pparen{T}^\times$ be its normalized derivative (cf. \Cref{def:51837}), and put $m\coloneqq v_T(h_P)$. %

	We first suppose that $s=e$. For every integer $k\geq0$, the $p^k$-fold iterate $F_\tau^{\circ p^k}\in\bfR_{\bfQ_p}$ commutes with $P$ and reduces modulo $\varpi_F$ to $f_{\tau,u}^{\circ p^k}=f_{\sigma_k,u}$, where $\sigma_k=\tau^{p^k}$. By \Cref{prop:78229}, the ramification numbers $d_k=i_{\bfE_{\bfQ_p}}(\sigma_k)$ satisfy
	$$d_k\equiv-1\pmod p \text{ and }  i_{\bfE_{\bfQ_p}}(\sigma_k^p)=d_{k+1}\geq(p+1)d_k.$$
	As a result, \Cref{thm:41759} applies to each $\sigma_k$ and provides a uniformizer $u_k$ of $\bfE_{\bfQ_p}$ in which
	$$f_{\sigma_k,u_k}=T+\alpha_kT^{d_k+1}+O\left(T^{2d_k+1}\right),\  \alpha_k\in\bfF_p^\times.$$
	Choose a coefficientwise lift $H_k\in T(\calO_F/\varpi_F^{e+1}\calO_F)\bbrac{T}$ of $\ev_u^{-1}(u_k)\in T\bfF_p\bbrac{T}$, and set
	$$P_k\coloneqq H_k\circ P\circ H_k\cinv,\  G_k\coloneqq H_k\circ F_\tau^{\circ p^k}\circ H_k\cinv.$$
	By \Cref{prop:71074} and \Cref{lem:84891}, $P_k$ and $G_k$ commute in $\bfR_{\bfQ_p}$, $\overline{P_k}=T^p$, $\overline{G_k}=f_{\sigma_k,u_k}$, $s(P_k)=s(P)=e$, and $v_T(h_{P_k})=v_T(h_P)=m$. Applying the valuation bound of \Cref{thm:10556} to the pair $(P_k,G_k)$ for every $k\geq 0$ yields
	$$m\leq p-1-d_k\ \text{for all }k\geq 0.$$
	This is impossible, since $m$ is fixed while $\lim_{k\to\infty}d_k=+\infty$ by \Cref{prop:78229}.%

	Now we assume that $s<e$. By \Cref{prop:78229}, we have $i_{\bfE_{\bfQ_p}}(\tau)=d_0=2p-1\geq 1$, so $f_{\tau,u}\in T+T^2\bfF_p\bbrac{T}$. By \Cref{lem:99761}, the linear coefficient of $f_{\delta,u}$ has multiplicative order $p-1$ in $\bfF_p^\times$. As a result, \Cref{lem:45874} produces a nonzero exact differential $\omega\in\bfF_p\pparen{T}\d T$ that is fixed by $\tau$. On the other hand, \Cref{prop:78229} shows that $\tau$ acts with infinite order on $\bfE_{\bfQ_p}$ and that the ramification numbers $d_k=i_{\bfE_{\bfQ_p}}\left(\tau^{p^k}\right)$ converge $p$-adically to $(p-1)/(p+1)$. Applying \Cref{prop:15106} to $\sigma=\tau$ and $\omega$ then forces
	$$\frac{p-1}{p+1}=-\opn{ord}_T(\omega)-1\in\bfZ,$$
	which is impossible and completes  the proof of \Cref{thm:42866}.
\end{proof}

\begin{proof}[Proof of \Cref{thm:84624}]
	Let $F/\bfQ_p$ be a finite totally ramified extension. If the action of $\Gamma$ on $X_{\bfQ_p}(\bfQ_p^{\FT})$ were $F$-liftable in the sense of \Cref{def:49607}, there would exist $P$ and $(F_g)_{g\in\Gamma}$ in $\bfA_{\bfQ_p}$ satisfying $\overline P=T^p$, $\ev_u(\overline{F_g})=g\ast u$, and $F_g\circ P=P\circ F_g$ for all $g\in\Gamma$. Reducing $P,F_\tau,F_\delta$ modulo $\varpi_F^{e+1}$ would then yield elements of $\bfR_{\bfQ_p}$ contradicting \Cref{thm:42866}.
\end{proof}
\printbibliography
\end{document}